\documentclass[12pt]{amsart}

\usepackage{bm}
\usepackage{fullpage}
\usepackage{amssymb}
\usepackage{amsmath,amsfonts,amsthm}
\usepackage{color}

\newtheorem{alphtheorem}{Theorem}

\newtheorem{problem}{Problem}

\newtheorem{lemma}{Lemma}

\newtheorem{claim}{Claim}

\newtheorem{theorem}{Theorem}

\newenvironment{changemargin}[2]{\begin{list}{}{%
\setlength{\topsep}{0pt}%
\setlength{\leftmargin}{0pt}%
\setlength{\rightmargin}{0pt}%
\setlength{\listparindent}{\parindent}%
\setlength{\itemindent}{\parindent}%
\setlength{\parsep}{0pt plus 1pt}%
\addtolength{\leftmargin}{#1}%
\addtolength{\rightmargin}{#2}%
}\item }{\end{list}}

\def\S{\mathcal{S}}

\def\A{\mathcal{A}}
\def\B{\mathcal{B}}
\def\C{\mathcal{C}}

\def\KG{\operatorname{KG}}

\title{ Extremal $G$-free induced subgraphs of Kneser graphs}
\author{Meysam Alishahi}
\address{M. Alishahi, 
Faculty of Mathematical Sciences,
Shahrood University of Technology, Shahrood, Iran}
\email{meysam\_alishahi@shahroodut.ac.ir}

\author{Ali Taherkhani}
\address{A. Taherkhani, 
Department of Mathematics, Institute for Advanced Studies in Basic Sciences (IASBS), Zanjan 45137-66731, Iran}
\email{ali.taherkhani@iasbs.ac.ir}

\begin{document}
\maketitle 

\begin{abstract}
The Kneser graph $\KG_{n,k}$ is a graph whose vertex set is the family of all $k$-subsets of $[n]$ and two vertices are adjacent if their corresponding subsets are disjoint.
The classical Erd{\H o}s-Ko-Rado theorem determines the cardinality and structure of a maximum induced $K_2$-free  subgraph in $\KG_{n,k}$. 
As a generalization of the Erd{\H o}s-Ko-Rado theorem,  Erd{\H o}s proposed a conjecture about the maximum order of an induced $K_{s+1}$-free subgraph of $\KG_{n,k}$.  As the best known result concerning this conjecture,  
Frankl~[{\it Journal of Combinatorial Theory, Series A, {\rm 2013}}], when $n \geq(2s+1)k-s$, gave an affirmative answer to this conjecture and also determined the structure of such a subgraph. In this paper, generalizing the Erd{\H o}s-Ko-Rado theorem and the Erd{\H o}s matching conjecture, we consider the problem of determining the structure of a maximum family $\A$ for which $\KG_{n,k}[\A]$ has no subgraph isomorphic to a given graph $G$. 
In this regard, we determine the size and the structure of such a family provided that $n$ is sufficiently large with respect to  $G$ and $k$. 
Furthermore, for the case $G=K_{1,t}$, we present a Hilton-Milner type theorem regarding above-mentioned problem, which specializes to an improvement of a result by Gerbner {\it et al.}~[{\it  SIAM Journal on Discrete Mathematics, {\rm 2012}}]. \\ \\
\end{abstract}

\section{Introduction and Main Results}

Let $n$ and $k$ be two positive integers such that $n\geq k$. Throughout the paper, the two symbols $[n]$ and ${[n]\choose k}$, respectively, stand for the sets 
$\{1,\ldots,n\}$ and $\{A\subseteq[n]\colon |A|=k\}$.  The {\it Kneser graph $\KG_{n,k}$} is a graph whose vertex set is ${[n]\choose k}$ 
where  two vertices are adjacent if their corresponding sets are disjoint. Kneser graphs are important objects in many areas of combinatorics and they are involved in many 
interesting results, for instance see~\cite{EKR61,HilMil67,Kneser55,Lovasz78,Matousek04,Schrijver78}. A family $\A\subseteq {[n]\choose k}$ is called {\it intersecting} whenever  
 each of two members of $\A$ have nonempty intersection , i.e.,
 $\A$ is an independent set in $\KG_{n,k}$.
Let us recall that for a graph $G$, the set $U\subseteq V(G)$ is independent if $G[U]$, the {\it subgraph induced by $U$}, has no edge. Moreover, the size of a maximum independent set 
in $G$ is called  the {\it independence number of $G$} and is denoted by $\alpha(G)$. Clearly, for each $i\in[n]$, the set 
$$\S_i=\Big\{A\in{[n]\choose k}\colon i\in A \Big\}$$ is an independent set of $\KG_{n,k}$. An independent set $\A$ of $\KG_{n,k}$ is called a {\it star } if  
$\A\subseteq \S_i$ for some $i\in[n]$. Otherwise, $\A$ is called a {\it nontrivial independent set}. 
For an $\ell$-subset $L\subseteq [n]$, the union of $\ell$ stars $\bigcup\limits_{i\in L}\S_{i}$ is said to be an {\it $\ell$-constellation} and is denoted by $C(L)$, i.e.,
$$C(L)=\Big\{A\in {[n]\choose k}\colon A\cap L\neq\varnothing\Big\}.$$ 
The well-known Erd{\H o}s-Ko-Rado theorem~\cite{EKR61} asserts that $\alpha(\KG_{n,k})={n-1\choose k-1}$ provided that $n\geq 2k$; furthermore, 
if $n>2k$, then the only independent sets of this size are maximal stars. 

Evidently, an independent set in a graph $G$ is a subset $U\subseteq V(G)$ for which the induced subgraph $G[U]$ contains no subgraph isomorphic to $K_2$. 
This leads to a number of interesting generalizations of the classical concepts in graph theory, such as independent sets and also proper vertex-coloring of graphs 
(see~\cite{DarFra95, harary1}).  In this regard, we are interested in the following problem which mainly motivated the present paper. 
\begin{problem}\label{problem}
Given a graph $G$, how large a subset $\A\subseteq {[n]\choose k}$ must be chosen to guarantee that $\KG_{n,k}[\A]$ has some subgraph isomorphic to $G$? 
Also, what is the structure of the largest subset $\A\subseteq {[n]\choose k}$ for which   $\KG_{n,k}[\A]$ has no subgraph isomorphic to $G$?
\end{problem} 

This problem has been investigated for some specific  graphs $G$ in the literature, but in a different language.  
In the following, we list some of them.\\
\begin{itemize}
\item The Erd{\H o}s-Ko-Rado theorem answers to the problem when $G=K_2$.\\

\item The case $G=K_{1,t}$ is studied by Gerbner {\it et al.} in~\cite{MR3022158}. For $t=2$, they have proved that if $n\geq 2k+2$, then any maximum subset $\A\subseteq {[n]\choose k}$ for which $\KG_{n,k}[\A]$ has no subgraph isomorphic to $K_{1,2}$ has the cardinality at most ${n-1\choose k-1}$ and the equality holds if and only if  $\A$ is a star. Moreover, for an arbitrary  $t$, they proved the same assertion provided that $n\geq \min\{O(tk^2),O(k^3+tk)\}$. 

\item The case $G=K_{s,t}$ was studied by Katona and Nagy in~\cite{MR3386026}. They proved that 
there is a threshold $n(k,t)$ such that for $n\geq n(k,t)$, if
$\A\subseteq{[n]\choose k}$ and $\KG_{n,k}[\A]$ is a $K_{s,t}$-free graph, then $|\A|\leq {n-1\choose k-1}+s-1$.\\

\item The Erd{\H o}s Matching Conjecture~\cite{Erdos65} suggests the exact value for the size of a maximum subset $\A\subseteq {[n]\choose k}$ for which $\KG_{n,k}[\A]$ has no subgraph isomorphic to $K_{s+1}$. This conjecture has been studied extensively in the literature. It has been already proved that the conjecture is true for $k\leq 3$  (see~\cite{ErdGal59,Frankl2017,LucMie14}). Also, improving the earlier results
 in~\cite{BolDayErd76,Erdos65,FranLucMie12,HuaLohSud12}, Frankl~\cite{Frankl13} confirmed the conjecture for $n\geq (2s+1)k-s$; moreover, he determined the structure of $\A$ in the extremal case.
\end{itemize}
Scott and Wilmer~\cite{ScottWilmer14} studied the number of vertices and the structure of an induced Kneser subgraph whose vertex degrees are located in an interval. 
Another interesting generalization of the Erd{\H o}s-Ko-Rado theorem can be found in~\cite[Theorem~3]{HavWood2014} which concerns the maximum number of vertices for a multipartite subgraph of the complement of Kneser graphs which has bounded size in its parts. 

\begin{alphtheorem}{\rm (Frankl~\cite{Frankl13})}\label{thm:Frankl}
Let $\A\subseteq {[n]\choose k}$ such that $\KG_{n,k}[\A]$ has no-subgraph isomorphic to $K_{s+1}$ and $n\geq(2s + 1)k-s$. Then 
$|\A|\leq {n\choose k}-{n-s\choose k}$
with equality if and only if $\A$ is equal to an $s$-constellation $C(L)$ for some $L\subseteq [n]$. 
\end{alphtheorem}

In the present paper, along with some other auxiliary results, we will answer to Problem~\ref{problem} provided that $n$ is sufficiently large. 
Before stating our main results, we need to introduce some preliminaries. For a given graph $G$ with $\chi(G)=q$, the minimum size of a color class 
among all proper $q$-coloring of $G$ is denoted  by $\eta(G)$, i.e.,
$$\eta(G)=\min\left\{\min\limits_{i\in[q]}|U_i|\colon (U_1,\ldots,U_{q})\text{ is a proper $q$-coloring of } G\right\}.$$
 A subgraph of $G$ is called {\it special} if removing its vertices from $G$ reduces the chromatic number by one. 
Note that $G$ is a subgraph of a complete $q$-partite 
graph $K_{t_1\ldots,t_q}$ for some positive integers $t_1,\ldots, t_q$, where $t_1\geq\cdots\geq t_q=\eta(G)$. 
Now, we are in a position to state our first main result. 
\begin{theorem}\label{thm:main}
Let $k\geq 2$ be a fixed positive integer and $G$ be a fixed graph with $|V(G)|=m$, $\chi(G)=q$ and $\eta(G)=\eta$. 
There exists a threshold $N(G,k)$ such that 
for any $n\geq N(G,k)$ and for any  $\A\subseteq {[n]\choose k}$, if $\KG_{n,k}[\A]$ has no subgraph  isomorphic to $G$, then 
$$|\A|\leq {n\choose k}-{n-q+1\choose k}+\eta-1.$$ 
Moreover, the equality holds if and only if there is a $(q-1)$-set $L\subseteq [n]$ such that 
$$|\A\setminus C(L)|= \eta-1$$
 and $\KG_{n,k}[\A\setminus C(L)]$ has no subgraph isomorphic to  a special subgraph of $G$. 
\end{theorem}

Let $s$ and $t$ be two positive integers such that $t\geq s$. 
A subset $\A\subseteq {[n]\choose k}$
is said to be {\it $(s,t)$-union intersecting}  whenever for each $\{A_1,\ldots,A_s\},\{B_1,\ldots,B_t\}\subseteq \A$, we have 
$$\left(\bigcup\limits_{i=1}^s A_i\right)\cap\left(\bigcup\limits_{i=1}^tB_i\right)\neq \varnothing.$$
It is clear that the Erd{\H o}s-Ko-Rado theorem determines the maximum possible size of a $(1,1)$-union intersecting family $\A\subseteq {[n]\choose k}$.
From this point of view, one may naturally ask for the maximum size of an $(s,t)$-union intersecting family $\A\subseteq {[n]\choose k}$.
Note that any family  $\A\subseteq {[n]\choose k}$ with the property that $\KG_{n,k}[\A]$ has no subgraph isomorphic to $K_{s,t}$ is an
$(s,t)$-union intersecting family and vice versa. This observation implies that if we set $G=K_{s,t}$, then Theroem~\ref{thm:main} gives a partial answer to the aforementioned question. 
However, in the next theorem we prove a better result by estimating the threshold $N(G,k)$ appearing in the statement of Theorem~\ref{thm:main}.
 It should be  mentioned that  the maximum size
of an $(s,t)$-union intersecting family $\A\subset {[n]\choose k}$ have been investigated by Katona and Nagy  in~\cite{MR3386026}. They prove that for sufficiently large $n$,
every $(s,t)$-union intersecting family $\A\subset {[n]\choose k}$ has the cardinality at most $ {n-1\choose k-1}+s-1$.  
In the next theorem, improving this result,  we show that all $(s,t)$-union intersecting families with the maximum possible size are formed by the union of a maximal star and a set of  $s-1$ $k$-sets.

\begin{theorem}\label{thm:EKR}
Let $k,t$ and $s$ be fixed positive integers such that $k\geq 2$ and $t\geq s\geq 1$.  
If $n\geq 1+\max \{2(sk(k-1)+t-1), 2^{2s(1+{3\over 2k-2})}(t-1)^{1\over k-1}(k-1) \}$, then  
 any $(s,t)$-union intersecting family  $\A\subseteq{[n]\choose k}$ has the cardinality at most 
${n-1\choose k-1}+s-1$. Moreover, the equality holds if and only if there is an $i\in[n]$ such that $\A$ is a union of the star $\S_i$ and $s-1$ vertices
from $\A\subseteq{[n]\choose k}\setminus \S_i$. 
\end{theorem}
Note that if we set $t=1$, then the previous theorem is an immediate consequence of  the Erd{\H o}s-Ko-Rado theorem.

Extending the Erd{\H o}s-Ko-Rado theorem, Hilton and Milner~\cite{HilMil67} determined, when $n>2k$, the maximum possible size of an independent set in $\KG_{n,k}$ which is contained in no star, i.e., 
the independent set is {\it nontrivial}. 
In detail, Hilton and Milner proved that for $n>2k$, any nontrivial independent set has the cardinality at most ${n-1\choose k-1}-{n-k-1\choose k-1}+1$. 
Many other interesting extensions of the Erd{\H o}s-Ko-Rado theorem and the Hilton-Milner theorem have since been proved,  for instance see~\cite{MR3482268,MR3403515,DeFr1983,MR0480051,MR1415313,Frankl13,FranLucMie12,FRANKL20121388,MR3022158,MR2489272,KATONA1972183,MR2202076,MR2285800,MR771733}.

When $s=1$ and $t\geq 2$, we have the following theorem giving a result stronger than the previous theorem. 
\begin{theorem}\label{HilMilnew}
Let $n$, $k$ and $t$ be positive integers such that $k\geq 2$, and $t\geq 2$. 
Any $(1,t)$-union intersecting family $\A\subseteq {[n]\choose k}$ of size at least $${n-1\choose k-1}-{n-k-1\choose k-1}+(t-1){2k-1\choose k-1}+t$$
is contained in some star $\S_i$. In particular, if  $n\geq {3k\over2}(1+(t-1+{t\over {2k-1\choose k-1}})^{1\over{k-1}})$ and
$\A$ is $(1,t)$-union intersecting, then $\A\leq{n-1\choose k-1}$ and the equality holds 
if only if $\A$ is the same as a star $\S_i$ for some $i\in [n].$
\end{theorem}
As mentioned before, the maximum size of a $(1,t)$-union intersecting family $\A\subseteq {[n]\choose k}$ has been already studied 
by Grebner {\it et al.} in~\cite{MR3022158}.
They proved  that there exists  an $N=N(k,t)$ such that if $n\geq N$ then the size of a (1,t)-union intersecting family $\A\subseteq {[n]\choose k}$ is at most ${n-1\choose k-1}$
with equality if and only if $\A$ is equal to some star $\S_i$. They showed that $N(k,t)\leq \min\{O(tk^2),O(k^3+tk)\}$ and posed the problem  of finding the smallest value of 
$N(k,t)$. Also, with an interesting proof, using Katona's cycle method, they showed that for  $t=2$ and $k\geq 3$,
the minimum of $N(k,2)$ is $2k+2$. Theorem~\ref{HilMilnew} implies that $N(k,t)\leq {3k\over2}(1+(t-1+{t\over {2k-1\choose k-1}})^{1\over{k-1}})$
which is an improvement of their result when $t\geq 3$.

For integers $n$ and $k$ with $n> 2k$, we set $M={n-1\choose k-1}-{n-k-1\choose k-1}$. 
Note that $M\leq k{n-2\choose k-2}=O(n^{k-2}).$ 
In the next theorem, we present a Hilton-Milner type theorem improving Theorem~\ref{thm:EKR}. 

\begin{theorem}\label{thm:hiltontype}
Let $k, s$ and $t$ be three integers such that $k\geq 3$ and $t\geq s\geq 1$ and also let $\beta$ be a  positive real number. 
There is an integer $N=N(s,t,\beta)$ such that if $n\geq N$, then for  
any $(s,t)$-union intersecting family $\A\subseteq {[n]\choose k}$ with $|\A|\geq (s+\beta)M$, we have $\ell(\A)\leq  s-1$. 
\end{theorem}
If we set $s=t=1$, then the previous theorem implies that any intersecting family $\A\subseteq {[n]\choose k}$ with $|\A|\geq (t+\beta)M$ is contained in a unique star $\S_i$ 
provided that $n$ is large enough. However, the original Hilton-Milner theorem asserts much more; for $n>2k$, any intersecting family $\A\subseteq {[n]\choose k}$ with $|\A|\geq M+2$ 
is contained in a unique star $\S_i$. 

\section{Proof of Main Results}
\subsection{Preliminaries}
For a given graph $G$, the {\it Tur\'an number ${\rm ex}(n,G)$} is the maximum possible number of edges for a spanning subgraph of $K_n$ with no subgraph isomorphic to $G$.  
The well-known Erd{\H o}s-Stone-Simonovits theorem~\cite{MR0205876,MR0018807} asserts that for any graph $G$ with $\chi(G)\geq 2$,
  \begin{equation}
{\rm ex}(G,n)=(1-{1\over \chi(G)-1}){n\choose 2}+o(n^2).
\end{equation}
This equality asymptotically determines ${\rm ex}(G,n)$ provided that $\chi(G)\geq 3$. However, when $\chi(G)=2$, it just indicates that ${\rm ex}(G,n)=o(n^2)$ and gives no more information about 
${\rm ex}(G,n)$. Indeed, for the case that $G$ is bipartite, the order of ${\rm ex}(G,n)$ is open in general.
In this regard, K{\H o}vari, S\'os, and Tur\'an~\cite{Kovari1954} proved that if $1\leq s\leq t$, then  ${\rm ex}(n,K_{s,t})\leq c(s)n^{2-{1\over s}}$.  
In what follows, we review the proof of this result since we want to estimate the constant $c(s)$. 
Assume that $G$ is an $n$-vertex graph having no $K_{s,t}$ subgraph. Clearly, we must have 
$$\sum\limits_{v\in V(G)}{d(v)\choose s}\leq (t-1){n\choose s}.$$
Consequently, since $\displaystyle n{{2|E(G)|\over n}\choose s}=n{{1\over n}\sum\limits_{v\in V(G)} \deg(v)\choose s}\leq \sum\limits_{v\in V(G)}{\deg(v)\choose s}$, we have 
$$n{({2|E(G)|\over n}-(s-1))^s\over s!}\leq n{{2|E(G)|\over n}\choose s}\leq (t-1){n\choose s}\leq (t-1){n^s\over s!},$$
which implies 
\begin{equation}\label{bipartiteturan}|E(G)|\leq ({1\over 2}+{s-1\over{n^{1-{1\over s}}}}){\sqrt[s]{t-1}}n^{2-{1\over s}}.
\end{equation}

\subsection{Proofs}
This section is devoted to the proofs of the results stated in the previous section. 
For a family $\A\subseteq {[n]\choose k}$, we define $\A^*$ to be the largest intersecting subfamily of $\A$; 
if there are more than one such a subfamily, we just choose one of them. Now, we set $\ell(\A)=|\A\setminus\A^*|$. 
Also, we define $\A^*_i$ to be equal to $\A\cap\S_i$. 
\begin{proof}[Proof of Theorem~\ref{thm:main}]  
Set $T(k,q,\eta)=  {n\choose k}-{n-q+1\choose k}+\eta$.
To prove the first part of the theorem, it suffices to show that for any  $\A\subseteq {[n]\choose k}$, 
if $|\A|\geq T(k,q,\eta)$, then $\KG_{n,k}[\A]$ has some subgraph isomorphic to $G$. 
We proceed by induction on $q$.
If $q=1$, then clearly the assertion holds. Now, we assume that $q\geq 2$. 
Without loss of generality, we can assume that $G=K_{t_1,\ldots,t_q}$, for some positive integers $t_1,\ldots,t_q$, where $t_1\geq\cdots\geq t_q=\eta$. 
Also, without loss of generality, assume that 
$|\A|=T(k,q,\eta)$. We distinguish the following different cases.
\begin{itemize}
\item[(I)] $\max\limits_{i\in[n]}|\A^*_i|< (\sum\limits_{i=2}^qt_i)M+t_1$.\\ In this case, $\KG_{n,k}[\A]$ has at least 
$${|\A|\choose 2}-\sum\limits_{i=1}^n{|\A^*_i|\choose 2}= 
{|\A|\choose 2}-o(|\A|^2)
$$
edges. 
Thus, by the Erd{\H o}s-Stone-Simonovits theorem, 
$\KG_{n,k}[\A]$ contains some subgraph isomorphic to $G$ provided that $n$ is large enough. \\
\item[(II)] $\max\limits_{i\in[n]}|\A^*_i|\geq (\sum\limits_{i=2}^qt_i)M+t_1$.\\ Without loss of generality, we can assume $|\A_n^*|\geq (\sum\limits_{i=2}^qt_i)M+t_1$. 
Consider $\A'=\A\setminus \A^*_n$. 
Since $|\A^*_n|\leq |\S_n|\leq {n-1\choose k-1}$, we have $|\A'|\geq T(k,q-1,\eta)$. Therefore, by the induction, 
there is a threshold $N(K_{t_2,\ldots,t_q},k)$ such that if $n\geq N(K_{t_2,\ldots,t_q},k)$, then 
$\KG_{n,k}[\A']$ contains a subgraph isomorphic to $K_{t_2,\ldots,t_q}$. 
Let $V\subseteq \A'$ be the vertex set of this subgraph. Note that each $A\in V$ is adjacent to at least $|\A_n^*|-M$ vertices in $ \A_n^*$. 
Consequently, there are at least $|\A_n^*|-(\sum\limits_{i=2}^qt_i)M\geq t_1$ vertices in $\A^*$ such that each of them is
adjacent to each vertex in $V$, completing the proof of the first part of the theorem.  
\end{itemize}
In what follows, we determine the structure of $\A$ when the equality holds.  This is done in two parts: 
We first prove the existence of the desired $L$ and then deduce the structure of $\A\setminus C(L)$. 
Suppose that $n\geq N(G,k)$, $|\A|=T(k,q,\eta)-1$, and $\KG_{n,k}[\A]$ has no subgraph isomorphic to $G$. 
In what follows, we prove that there is a $(q-1)$-set $L\subseteq [n]$ such that 
$$|\A\setminus C(L)|= \eta-1.$$
First note that we must have $\max\limits_{i\in[n]}|\A^*_i|\geq (\sum\limits_{i=2}^qt_i)M+t_1$. Otherwise, with the same argument as in the proof of Case~(I), 
$\KG_{n,k}[\A]$ contains some subgraph isomorphic to $G$ provided that $n$ is large enough, which is~not possible. 
Therefore, the argument employed in the proof of Case~(II) should fail. It can be verified that if we have $|\A'|\geq T(k,q-1,\eta)$, then the proof still works. 
Hence, we must have  $|\A^*_n|=\max\limits_{i\in[n]}|\A^*_i|={n-1\choose k-1}$. Note that $\A'\subseteq {[n-1]\choose k}$ has exactly 
$T(k,q-1,\eta)-1$ elements. Also, by following the proof of Case~(II), one can see that if $\KG_{n,k}[\A']$ has some subgraph isomorphic to $K_{t_2,\ldots,t_q}$,
then $\KG_{n,k}[\A]$ contains some subgraph isomorphic to $K_{t_1,\ldots,t_q}$. Consequently, 
$\KG_{n,k}[\A']$ has  no subgraph isomorphic to $K_{t_2,\ldots,t_q}$. 
Accordingly, by the induction, there is an $L\subseteq [n]$ of size $q-1$ such that $|\A\setminus C(L)|= \eta-1$. 
Suppose for the sake of contradiction that 
there is  a subset $\{B_1,\ldots,B_r\}\subseteq \A\setminus C(L)$ such that 
$\KG_{n,k}[\{B_1,\ldots,B_r\}]$ is isomorphic to a special subgraph $H$ of $G$.  
Let $(U_1,\ldots,U_{q-1})$ be a proper $(q-1)$-coloring of $G-H$ such that $|U_1|\geq \cdots\geq |U_{q-1}|$. 
For the simplicity of notation, set $L=\{1,\ldots,q-1\}$ and $b=|\bigcup\limits_{i=1}^rB_i|$.
Let $n$ be sufficiently large such that 
$$\displaystyle{\lfloor{n-b\over q-1}\rfloor \choose k-1}\geq \displaystyle{\lfloor{n-rk\over q-1}\rfloor \choose k-1}\geq |U_1|.$$
Now, consider $q-1$ pairwise disjoint subsets  $Q_1,\ldots,Q_{q-1}$ of $[n]\setminus (\bigcup\limits_{i=1}^r B_i)$ such that 
$|Q_i|=\lfloor{n-b\over q-1}\rfloor$ and  $i\in Q_i$, for each $i\in [q-1]$. 
Set $$\C=\{B_1,\ldots,B_r\}\cup\left(\bigcup\limits_{i=1}^{q-1} {Q_i\choose k-1}\right).$$
Now, one can simply check that $\KG_{n,k}[\C]$ contains $G$ as a subgraph, a contradiction. 

Suppose that there is a $(q-1)$-set $L\subseteq [n]$ such that $|\A\setminus C(L)|=  \eta-1$ and $\KG_{n,k}[\A\setminus C(L)]$ 
has no subgraph isomorphic to  a special subgraph of $G$. 
For a contradiction, assume that $G$ is a subgraph of $\KG_{n,k}[\A]$. 
Note that this subgraph must have some elements in  $\A\setminus C(L)$ as its vertices. 
Consider the subgraph of $G$ induced by these vertices. One can simply see that this subgraph is a special subgraph of $G$ which is~not possible. 
\end{proof} 
Frankl~\cite{MR670845} generalizing a classical result due to Bollob{\'a}s~\cite{MR0183653} proved the following theorem. 
\begin{alphtheorem}\label{skew}\cite{MR670845}
Let $\{(A_1,B_1),\ldots,(A_h,B_h)\}$ be a family of pairs of subsets of an arbitrary set with $|A_i|=k$ and $|B_i|=\ell$ for all $1\leq i\leq h$.
If $A_i\cap B_i=\varnothing$ for  $1\leq i\leq h$ and $A_i\cap B_j\neq\varnothing$ for  $1\leq i<j\leq h$, then
$h\leq {k+\ell\choose k}$.
\end{alphtheorem}
We will use this theorem to prove the following lemma which improves a result by Balogh {\it et al.}~\cite{MR3482268} asserting that for any $\A\subseteq{[n]\choose k}$,
the induced subgraph $\KG_{n,k}[\A]$ has at least ${\ell(\A)^2\over 2{2k\choose k}}$ edges. 
Our technique can be considered as a development of those in~\cite[Lemma 3.1]{MR3482268} and~\cite[Theorem 2.1]{ScottWilmer14}. 

\begin{lemma}\label{lem:1}
If $\A\subseteq {[n]\choose k}$, then 
$|E(\KG_{n,k}[\A])|\geq {\ell(\A)^2\over {2k\choose k}}.$
\end{lemma}
\begin{proof}
Note that if $\ell(\A)=0$, then there is nothing to prove. Henceforth, we assume that $\ell(A)>0$. 
Set $\A^1=\A$. Let $i\geq1$. For each $i\geq 1$, if $E(\KG_{n,k}[\A^{i}])\neq\varnothing$, we define
$\C^i,m_i,A^i$ and $\A^{i+1}$ as follows.\\
\begin{itemize}
\item Set  $\B^i_1=\A^{i}$.
\item Let $A^i_1B^i_1\in E(\KG_{n,k}[\B^i_1])$. 
\item While $E(\KG_{n,k}[\B^i_j\setminus N(A^i_j)])\neq\varnothing$, define $\B^i_{j+1}=\B^i_{j}\setminus N(A^i_{j})$ 
and choose $A^i_{j+1},B^i_{j+1}\in \B^i_{j+1}$ such that  $A^i_{j+1}B^i_{j+1}\in E(\KG_{n,k}[\B^i_{j+1}])$. 
\end{itemize}
Define $m_i$ to be the maximum index $j$ for which $E(\KG_{n,k}[\B^i_j])\neq\varnothing$ and
 $\C_i$ to be  the set of vertices in $\A^i$ which are adjacent to $A^i_j$ for some $j\in[m_{i}]$. 

Clearly, in view of how the $(A^i_j,B^i_j)$'s are chosen, we have $E(\KG_{n,k}[\B^i_{m_i}\setminus N(A^i_{m_i})])=\varnothing$.
 Hence, $$\A^i\setminus \C^i=\B^i_{m_i}\setminus N(A^i_{m_i})$$
 is an independent  set of $\KG_{n,k}$. Accordingly, 
$|\C^i|\geq \ell(\A^i)$. Now, it is clear that there is a vertex $A^i\in \{A^i_1,\ldots,A^i_{m_i}\}$ such that 
$$\deg(A^i)\geq {|\C^i|\over m_i}\geq {\ell(\A^i)\over m_i}.$$
If $E(\KG_{n,k}[\A^{i}\setminus \{A^{i}\}])\neq\varnothing$, then define 
$\A^{i+1}=\A^{i}\setminus \{A^{i}\}$.

Let $p$ be the maximum index for which $\A^p$ is defined, i.e. $E(\KG_{n,k}[\A^p\setminus \{A^p\}])=\varnothing$. 
Note that $p\geq \ell(\A)$. 
In view of the definition of $A^i$'s, we have  
$$\begin{array}{lll}
|E(\KG_{n,k}[\A])| & \geq & \sum\limits_{i=1}^{p} {\ell(\A_i)\over m_i}\\
			   & \geq & \sum\limits_{i=1}^{\ell(\A)} {\ell(\A)-i+1\over m}\\
			   & \geq & {\ell(\A)(\ell(\A)+1)\over 2m},
\end{array}$$
where $m=\max_{i\geq1} m_i$. 
To finish the proof, it is enough to show that $m \leq {2k-1\choose k-1}$. 
For  $m_i+1\leq j\leq 2m_i$, set $A^i_j=B^i_{2m_i-j+1}$ and $B^i_j=A^i_{2m_i-j+1}$. One can verify that  the family 
$$\left\{(A^i_1,B^i_1),\ldots,(A^i_{2m_i},B^i_{2m_i})\right\}$$ 
satisfies the condition of Theorem~\ref{skew}. 
Consequently, we must have $2m_i\leq {2k\choose k}$ or equivalently, $m_i\leq {2k-1\choose k-1}$. 
Therefore, for each $i$, we have $m_i\leq {2k-1\choose k-1}$, as requested. 
\end{proof}
 We are now in a position to present the proofs of Theorems~\ref{thm:EKR},~\ref{HilMilnew} and~\ref{thm:hiltontype}. 
 \begin{proof}[Proof of Theorem~{\rm\ref{thm:EKR}}]
The case $t=1$ has been already known to be true by the Erd{\H o}s-Ko-Rado theorem. 
For $t\geq 2$, 
we prove the following statement which implies the theorem, immediately. \\
\begin{changemargin}{1cm}{1cm}
\noindent{\it
Under the assumptions of Theorem~{\rm\ref{thm:EKR}}, if $\A\subseteq {[n]\choose k}$ is an $(s,t)$-union intersecting family where $t\geq 2$ and $|\A|\geq {n-1\choose k-1}+s-1$, then 
$|\A|= {n-1\choose k-1}+s-1$ and 
$\ell(\A)= s-1$.  In other words, $\A$ is the union of a star $\S_i$ and $s-1$ vertices from  ${[n]\choose k}\setminus \S_i$.}
\end{changemargin} \vspace{.5cm}
First note that by the Erd{\H o}s-Ko-Rado theorem, we already have $\ell(\A)\geq s-1$. 
Furthermore, if we prove that  $\ell(\A)= s-1$, then it implies $|\A|= {n-1\choose k-1}+s-1$ as well. 
For a contradiction, suppose that $\ell(\A)\geq s.$ 
Without loss of generality, we can assume that $|\A|={n-1\choose k-1}+s-1$. To see this, if $|\A^*|={n-1\choose k-1}$, then consider $\A'\subseteq \A$ such that $|\A^*\setminus \A'|=1$ and $|\A'|={n-1\choose k-1}+s-1$; otherwise, consider $\A'\subseteq \A$ such that $\A^*\subseteq \A'$ and $|\A'|={n-1\choose k-1}+s-1$.
Note that $\A'$ satisfies the conditions of aforementioned statement and we can thus work with $\A'$ instead of $\A$. 
For simplicity of notation, set $L={n-1\choose k-1}$, $M={n-1\choose k-1}-{n-k-1\choose k-1}$ and $R={2k\choose k}$. 

\begin{claim}\label{alphaKG[A]}
$\alpha(\KG_{n,k}[\A])\leq sM+t-1$. 
\end{claim}
\begin{proof}
For a contradiction, let $\B$ be an independent set of $\KG_{n,k}[\A]$ with $|\B|\geq sM+t$. 
In view of the Hilton-Milner theorem, there is some $i$ for which $\B \subseteq \S_i$. Now, we consider $C_1,\ldots,C_s\in \A\setminus \B$ (this is  possible since $\ell(\A)\geq s$).
Now, it is clear that there are at least $t$ elements $B_1,\ldots,B_t$ in $\B$ such that  each of  $B_i$'s is adjacent to each of 
$C_j$'s contradicting the fact that $\A$ is $(s,t)$-union intersecting. 
\end{proof}
In view of Claim~\ref{alphaKG[A]}, 
$\ell(\A)= |\A|-\alpha(\KG_{n,k}[\A])\geq L-sM-t+1$. 
Note that 
$$\begin{array}{lll}
L-sM-t+1&\geq& L-sk{n-2\choose k-2}-t+1\\
&&\\
&=&\Big(1-{sk(k-1)\over n-1}-{t-1\over L}\Big)L. 
\end{array}$$
Using Lemma~\ref{lem:1} and the fact that $R={2k \choose k}\leq 2^{2k-1}$, we have 
$$|E(\KG_{n,k}[\A])|\geq {\ell(\A)^2\over R}\geq{(1-{sk(k-1)\over n-1}-{t-1\over L})^2L^2\over R}\geq{(1-{sk(k-1)\over n-1}-{t-1\over L})^2L^2\over 2^{2k-1}}.$$ 
On the other hand, in view of the discussion before the statement of Theorem~\ref{thm:EKR}, we know that   $\KG_{n,k}[\A]$ contains no $K_{s,t}$ subgraph. 
Consequently, Inequality~\ref{bipartiteturan} implies that 
	$$|E(\KG_{n,k}[\A])|\leq ({1\over 2}+{s-1\over{L^{1-{1\over s}}}}){\sqrt[s]{t-1}}L^{2-{1\over s}}\leq {\sqrt[s]{t-1}}L^{2-{1\over s}}.$$
Thus, we should have 
	$$({1-{sk(k-1)\over n-1}-{t-1\over L}})^2\leq {2^{2k-1}} {\sqrt[s]{t-1}\over L^{1\over s}},$$
which implies  	
	$${1-{sk(k-1)\over n-1}-{t-1\over L}}\leq {2^{k-{1\over 2}}} {\sqrt[2s]{t-1}\over L^{1\over 2s}}.$$
 Accordingly, if $n-1\geq \max \{2(sk(k-1)+t-1), 2^{2s(1+{3\over 2k-2})}(t-1)^{1\over k-1}(k-1) \}$, then 
the left hand side of inequality is at least $1\over 2$ and moreover by using the inequality $L={n-1\choose k-1}\geq ({n-1\over k-1})^{k-1}$ the right  hand side is less than half, a contradiction. 
\end{proof}

\begin{proof}[Proof of Theorem~{\rm\ref{HilMilnew}}]
For a contradiction, suppose that $\A$ is not contained in any star. Since $|\A|\geq {n-1\choose k-1}-{n-k-1\choose k-1}+2$, there must be some disjoint 
pair in $\A$, i.e., $\KG_{n,k}[\A]$ has at least one edge. Choose $A_1\in \A_1=\A$ and $B_1\in N(A_1)$. 
For each $i\geq 2$, set $\A_i=\A_{i-1}\setminus N(A_{i-1})$ and until $E(\KG_{n,k}(\A_{i}))\neq\varnothing$,
choose $A_i\in \A_{i}$ and $B_i\in N(A_i)$. Let $m$ be the largest index $i$ for which $E(\KG_{n,k}(\A_{i}))\neq\varnothing$.
For  $m+1\leq j\le 2m$, set $A_j=B_{2m_i-j+1}$ and $B_j=A_{2m_i-j+1}$. 
It can be verified that $\{(A_1,B_1),\ldots,(A_{2m,}B_{2m})\}$
satisfies the condition of Theorem~\ref{skew} for $l=k$. 
Therefore, we must have $2m\leq {2k\choose k}$ and consequently $m\leq {2k-1\choose k-1}$. 
Set $$\C=\Big\{C\in\A| {\rm there\,\, is\,\, some\,\,} i\leq m\,\,{\rm such\,\, that}\,\, C\cap A_i=\varnothing\Big\}.$$ In other words, $\C$ is the set of all neighbors of $\{A_1,\ldots,A_m\}$.  
Since $\Delta(\KG_{n,k}[\A])\leq t-1$, we have 
$|\C|\leq (t-1)m.$
Note that  $\A\setminus \C$ in an independent set of $\KG_{n,k}$. 
Indeed, since $m$ is the largest index for which  $E(\KG_{n,k}(\A_{i}))\neq\varnothing$, we must have 
$E(\KG_{n,k}(\A_{m}\setminus N(A_m)))=\varnothing$. But, in view of the definition of $A_i$'s,  it is clear that $\A\setminus \C=\A_{m}\setminus N(A_m)$.  

If $|{ \A\setminus \C}|\geq {n-1\choose k-1}-{n-k-1\choose k-1}+t$, then $B_1$ has at least $t$ neighbors in $\A\setminus \C$, which is not possible.
Accordingly, $|{ \A\setminus \C}|\leq {n-1\choose k-1}-{n-k-1\choose k-1}+t-1$ and consequently, 
$$
\begin{array}{lll}
|\A|	       & =      &  |\C|+|\A\setminus \C|\\
   	       & \leq & (t-1) {2k-1\choose k-1}+{n-1\choose k-1}-{n-k-1\choose k-1}+t-1,
\end{array}$$
which is impossible. 

First, note that $n\geq {3k\over2}(1+(t-1+{t\over {2k-1\choose k-1}})^{1\over{k-1}})$. 
To complete the proof, it suffices to show that 
$${n-1\choose k-1}-{n-k-1\choose k-1}+(t-1){2k-1\choose k-1}+t\leq {n-1\choose k-1},$$
or equivalently, 
$$t-1+{t\over {2k-1\choose k-1}}\leq\prod_{i=1}^{k-1}\frac{n-k-i}{2k-i}.$$
Since $\prod_{i=1}^{k-1}\frac{n-k-i}{2k-i}\geq (\frac{2n-3k}{3k})^{k-1}$,  the proof will be completed if 
$$t-1+{t\over {2k-1\choose k-1}}\leq(\frac{2n-3k}{3k})^{k-1}$$ which clearly holds. 
\end{proof}
\begin{proof}[Proof of Theorem~{\rm\ref{thm:hiltontype}}]
Without loss of generality, we may assume that $|\A|=\lceil (s+\beta)M\rceil$. 
We shall distinguish the following cases. 
\begin{itemize}
\item $\ell(\A)> \beta M-t$. 
	Using Lemma~\ref{lem:1}, we have $$|E(\KG_{n,k}[\A])|\geq {\ell(\A)^2\over R}=O(|\A|^2).$$ 
	Since every  $K_{s,t}$-free graph with $m$ vertices has at most $c(t)m^{2-{1\over t}}+o(m^{2-{1\over t}})$ edges, we deduce that $\KG_{n,k}[\A]$ 
	has some $K_{s,t}$ provided that $n$ is sufficiently large. \\

\item $s\leq \ell(\A)=|\A|-|\A^*|\leq\beta M-t$. It implies that $|\A^*|\geq |\A|-\beta M+t=sM+t$ and
	consequently, $\A^*$ is trivially intersecting (contained in a star). 
	Each $s$ vertices  in $\A-\A^*$ have at least  $|\A^*|-sM\geq t$ common neighbors in $\A^*$, 
	which completes the proof. 
\end{itemize}
\end{proof}
\section*{Acknowledgements}
We would like to thank Alex Scott for pointing out references~\cite{HavWood2014,ScottWilmer14}. 
\def\cprime{$'$} \def\cprime{$'$}

\end{document}